\documentclass[a4paper, 10pt
]{scrartcl}
\usepackage{amsmath}
\usepackage{amsfonts}
\usepackage{amssymb}
\usepackage{amsthm}
 \usepackage{verbatim}
\usepackage{picinpar}
\usepackage{dsfont}
\usepackage{setspace}
\usepackage[ansinew]{inputenc}
\newtheorem{Lem}{Lemma}
\newtheorem{Prop}[Lem]{Proposition}
\newtheorem{theorem}[Lem]{Theorem}
\newtheorem{Cor}[Lem]{Corollary}

\usepackage{dsfont}
\newcommand\tr{\mathop{\rm tr}\nolimits}

\newcommand{\R}{\mathbb{R}}

\newcommand{\N}{\mathbb{N}}

\newcommand{\M}{\mathcal{M}}
\newcommand{\RR}{\mathcal{R}}

\renewcommand{\L}{\mathcal{L}}
\newcommand{\A}{\textbf{A}}
\newcommand{\B}{\textbf{B}}
\renewcommand{\epsilon}{\varepsilon}
\begin{document}
 \title{Well posedness for a quasilinear generalisation of the matched microstructure model}
\author{Daniela Treutler}
\maketitle
\begin{abstract}
We consider a generalisation to quasilinear systems of the matched microstructure model. The proof of well posedness in a strong Sobolev setting is based on an approach via maximal regularity.
\let\thefootnote\relax\footnotetext{
\textbf{Mathematics Subject Classification.} AMS 35K55, 35K59}
\let\thefootnote\relax\footnotetext{\textbf{Keywords.} two-scale model, double porosity, quasilinear system }
\end{abstract}

\section{Introduction}
In this article we consider a quasilinear version of the matched microstructure model as it was introduced in \cite{ADH}. An approach to a quasilinear parabolic system was already described in \cite{ShoWa2} while the elliptic case is treated in \cite{ShoWa3}. The authors write down a weak formulation and use the methods of monotone operators. In contrast to that we aim to study strong solutions. In particular, a nonlinear generalisation of the linear model as it was given in \cite{Dani1} is given. The model is based on homogenization results. Moreover we consider a varying geometry in the microstructure. Recent developments for homogenization with varying cell structure have been obtained in \cite{MN1, MN2}. But the construction of function spaces therein differs from our ansatz and is based on \cite{MeiBo}. 
Let $\Omega\subset \R^n$ be a bounded smooth domain and for each $x\in \Omega$ assume that $\Omega_x$ is a given domain as in \cite{Dani1}. Then we investigate the well posedness of the system
\begin{align*}
 \left\{ \begin{array}{rll}
        \partial_t u - div_x( a_1(u) \nabla_x u) =& f_1(x,t,u, U), & x\in \Omega, t\in (0,T],\\
	u(x,t) =& 0, & x\in  \Gamma, t\in (0,T],\\
	u(t=0) =& u_0,\\
	\partial_t U - div_z(a_2(u) \nabla_z U) =& f_2(x,t,u, U), & z \in \Omega_x, x\in \Omega, t\in (0,T],\\
	U(x,z,t) =& u(x,t), & z\in \partial \Omega_x, x\in       \Omega, t\in (0,T],\\
	U(t=0) =& U_0.&  \\
       \end{array}
\right.
\end{align*}
To prove maximal regularity we use a generalisation of the definition of sectorial operators. 
The concept of $\RR$-boundedness and $\RR$-sectorial operators was invented in the last 20 years and is e.g. introduced 
by Denk, Hieber and Pr\"uss in \cite{DHP} or by Kunstmann and Weis in \cite{KuWei}.

In the second part of this work we discuss how we can allow $a_1$ and $a_2$ to depend on $U$. Since we work in an adapted $L_p-L_p$-setting we face some crucial problems to extend our methods to this case. But in the derivation of the model such a dependence appears naturally. It would be interesting for example to investigate the behaviour of non-Newtonian fluids. So we consider a special configuration and expand the result in \cite{ClLi} to show well posedness. 

\textbf{Notation} We use the notation as in \cite{Dani1}. Let $\Phi$ be the mapping from $\Omega$ and the unit ball $B = B(0,1) \subset \R^n$ to $\R^n\times \R^n$ such that it describes the cells $\Omega_x =\Phi(x,B)$. Further let $Q= \Phi(\Omega\times B)$. Let $\Phi_*, \Phi^*$ be the push forward and pullback operators. We will use the transformed function spaces
$$L_p(\Omega, W_p^s(\Omega_x)) = \Phi^* L_p(\Omega, W_p^s(B))$$ and the transformed trace $$\tr :L_p(\Omega, W_p^1(\Omega_x)) \to L_p(\Omega, W_p^{1-\frac1p}(\Omega_x))$$
with the corresponding right inverse $R$. Moreover we proved in \cite{Dani1} some interpolation results. Let $$D_0(u) = \{u\in L_p(\Omega, W_p^2(\Omega_x)); \tr U=u\}.$$ We write
\begin{align*}
Y_0&= L_p(\Omega)\times L_p(\Omega, L_p(\Omega_x)),\\
Y_1 &= dom(\A) = \bigcup_{u\in W_p^2(\Omega)\cap W_p^{1,0}(\Omega) } \{u\} \times D_0(u),\\
\textbf{A} (u,U)&=\left(-\Delta_x u, [\Phi_* \mathcal{A}_x \Phi^* \hat{U}(x)] \right),&& \text{for } (u,U)\in dom(\textbf{A}),
\end{align*}
where $\mathcal{A}_x$ is a transformation of the Laplace operator. Given $0<\Theta <1$, the real interpolation space has the form
\begin{equation}\label{Interpol} Y_{\theta, p}:= (Y_0,Y_1)_{\theta, p}  = \bigcup_{\begin{subarray}{l} u\in W_p^{2\theta} (\Omega)\\
 \quad \cap \ker \tr_\Gamma \end{subarray}}  \{u\} \times \{U\in L_p(\Omega, W_p^{2\theta}(\Omega_x)); \tr U = u\}.\end{equation}
Finally with $g(x)$ we denote the Riemannian metric on $B$ induced by the transformation.  
\section{Quasilinear Operators}
We use Nemytskii operators to generalize the above version of the operator $\A$. Given $2\leq p<\infty$, we prove maximal $L_p$-regularity in case that the initial data is sufficiently regular. This allows us to treat the original problem as an abstract quasilinear initial value problem. The proof of local existence then is based on a result of Cl\'ement and Li \cite{ClLi}.
First we define a Nemytskii operator $a_1$. Given $\tilde{a}_1 \in C^\infty(\overline{\Omega}\times \R)$, we set
\begin{align*}
&a_1(v)(x) := \tilde{a}_1(x,v(x)),\\
& a_1(v) \in C(\overline{\Omega}), \qquad \text{ for } v\in C(\overline{\Omega}).\end{align*}
Hence given $v\in C^1(\overline{\Omega})$, we define $A_1$ in the following way
\begin{align*}
 dom(A_1(v)) & = W_p^2(\Omega)\cap W_p^{1,0}(\Omega),\\
 A_1(v) u &= -div_x (a_1(v) \nabla_x u), && \text{for all }u \in \ dom(A_1(v)).
\end{align*}
Assume that there is $\eta >0$, such that
$$\tilde{a}_1(x,r) \geq \eta,  \qquad \text{ for all }(x,r) \in \overline{\Omega}\times \R.$$ 
Then $A_1(v)$ is a strongly elliptic operator for all $v\in C^1(\Omega)$. To define the operator acting in the cells we need a further function $\tilde{a}_2 \in C^\infty( Q\times \R)$, $r\in \R$. We set
\begin{align*}
 \tilde{b}_2(r)&:= \Phi^* \tilde{a}_2(r),\\
b_2(v)(x,y) &:= \tilde{b}_2((x,y), v(x)), \qquad \text{ for all }(x,y)\in \Omega\times B.
\end{align*}
Due to the assumptions on $\Phi$, there is $\alpha >0$, such that $\tilde{b}_2 \in C^{2+\alpha}(\Omega\times B\times \R)$ and thus $b_2\in C^1(\Omega\times B)$.
For any $(x,y) \in \Omega\times B$, $r\in \R_+$ assume that
$$ \tilde{b}_2((x,y),r) \geq \eta.$$
This leads to a strongly elliptic operator for fixed $x\in \Omega$. Let $\{g^{ij}(x)\}$ be the metric induced by the transformation of $B$ to $\Omega_x$. For $u\in W_p^2(B)$ we set
\begin{align*}
 \mathcal{A}_x(v(x)) u = -\frac{1}{\sqrt{|g(x)|}} \sum_{i,j} \partial_{y_i} \left( b_2(v)(x) \sqrt{|g(x)|} g^{ij}(x) \partial_{y_j} u\right). 
\end{align*}
Given $v\in C^1(\overline{\Omega})$, we define $\textbf{A}_2(v)$ by
\begin{align*}
  dom(\A_2(v)) & = \left\{U\in L_p(\Omega, W_p^2(\Omega_x)); \tr U =0 \right\},\\
 \A_2(v) U &= \Phi_* \left[ \mathcal{A}_x(v(x)) \hat{V}\right] && U \in dom(\A_2(v)).
\end{align*}
Here $V= \Phi^* U$ and $\hat{V}$ is a representative of $V$ such that $\hat{V}(x) \in W_p^2(B)$ for all $x\in \Omega$.
\begin{Lem}\label{Hilfe0}
 Let $v\in C^1(\Omega)$ be given. Then $\A_2(v)$ is a well defined, closed, densly defined, sectorial operator on $L_p(\Omega, L_p(\Omega_x))$.
\end{Lem}
\begin{proof}
Let $v\in C^1(\overline{\Omega})$ and $x\in \overline{\Omega}$ be given. First we consider the following operator $B_x$, 
\begin{align*}
 dom(B_x) &= W_p^2(B)\cap W_p^{1,0}(B)=: W_B,\\
B_x u &= \mathcal{A}_x(v(x)) u, && \text{for }u \in W_B.
\end{align*}
Then $B_x$ is a closed operator in $L_p(B)$. The function $v$ as well as $ \Phi(x)$ is in $C(\overline{\Omega})$. Hence the map $x\mapsto B_x$ is continuous from $\overline{\Omega}$ into $\L(W_B, L_p(B))$. Further the domain of definition of $B_x$ is the same for all $x\in \overline{\Omega}$. So Lemma 2 from \cite{Dani1} can be applied. It shows that the operator $\B$ defined by
\begin{align*}
 dom(\B) &= \{V\in L_p(\Omega, W_p^2(B)); \tr_S V=0\},\\
\B V &= \left[ \mathcal{A}_x(v(x)) \hat{V}(x)\right],
\end{align*}
is well defined and closed. The trace $\tr_S$ maps from $L_p(\Omega, W_p^1(B))$ to $L_p(\Omega, W_p^{1-\frac1p} (S))$. Again $\B$ is a densely defined operator. It holds
\begin{align}
 \B:= \Phi^* \A_2(v) \Phi_*.\label{25}
\end{align}
Since $b_2(v) \in C^{2+\alpha}(\overline{\Omega}\times B)$ and $\Phi^*$, $\Phi_*$ are bounded isomorphisms 
on the bounded domain $\Omega$ the moduli of continuity of the coefficients of highest order in $B_x$ are uniformly bounded. Hence a priori estimates (see \cite{GilbTru}) work for all $B_x$. The corresponding constants depend only on the dimension $n$, $p, \eta_2, \|\Phi_*\|, \|\Phi^*\|, \Lambda$, the domain
 $B$ and the moduli of continuity of the coefficients of the highest order term. So for all $x\in \overline{\Omega}$, there is a common sector $S_{\Theta,\omega}\subset \rho(-B_x)$ and there exists $M\geq 1$ such that
\begin{align*}
 \|(\lambda+ B_x)^{-1}\|_{\L(L_p(\B))} &\leq \frac{M}{|\lambda-\omega|}, && \text{for } \lambda \in S_{\Theta,\omega}, x\in \overline{\Omega}.
\end{align*}
Thus the assumptions of \cite{Dani1}, Lemma 3  are satisfied and hence $\B$ is sectorial. From (\ref{25}) and estimates for $\Phi^*, \Phi_*$,  we conclude the assertion.
\end{proof}
Remark that $A_1(v)$ is a sectorial operator in $L_p(\Omega)$. This follows from standard elliptic theory.
It is shown in \cite{DHP} that for operators on Banach spaces the properties to possess maximal $L_p$-regularity and 
$\RR$-sectoriality with an $\RR$-angle smaller than $\frac\pi2$ are equivalent. Further in this work the authors show maximal regularity for a certain class of elliptic operators on domains. This is used here. They impose several conditions. We will varify these constraints for $A_1(v)$ and $B_x$ ($x\in \Omega$) so we can later shift this to the full problem.
We show that indeed the operators defined above satisfy the smoothness (SC) and ellipticity conditions (EC) in \cite{DHP}.
\begin{Lem}\label{Hilfe1}
 Let $v\in C^1(\overline{\Omega})$. Take  $0<\epsilon<\frac{\pi}{2}$. Then for each $x\in \Omega$ the operator $B_x$ is $\RR$-sectorial of $\RR$-angle $\Phi_\mathcal{A}<\epsilon$. This means that for each $\Phi>\Phi_\mathcal{A}$ there is $\mu_\Phi\geq 0$ such that the parabolic initial-boundary value problem
\begin{align*}
 \partial_t u + \mu_\Phi u + B_xu &=f, &&t>0,\\
u(0)&= u_0
\end{align*}
has a unique solution in $L_p(\R_+, L_p(B)).$
\end{Lem}
\begin{proof} Let $v\in C^1(\overline{\Omega})$. Fix $x\in \Omega$.
By the assumptions $B_x$ fulfills (SC) and the ellipticity condition. So it remains to check whether the condition of Lopatinskii-Shapiro is satisfied.
This can be done using spherical coordinates. For details see \cite{Doktorarbeit}. Then Theorem 8.2 of \cite{DHP} proves the lemma.
\end{proof}
The uniform bounds for the transformation and $b_2$ ensure that $\mu_\Phi$ can be chosen independently of $x$. So we can take the term to the right hand side $f_2(u,U)$. Thus we may assume that for some fixed $\Phi<\frac{\pi}{2}$, it holds $\mu_\Phi=0$.  
\begin{Prop}\label{Hilfe2}
Let $v\in C^1 (\overline{\Omega })$. Then $\A_2(v)$ is $\RR$-sectorial with $\RR$-angle less or equal to $\Phi$.
\end{Prop}
\begin{proof}
(i) Let $v\in C^1 (\overline{\Omega })$ and $\{B_x, x\in  \Omega\}$ as before. We start to consider $\B$. From the proof of Lemma \ref{Hilfe0} we know that $\B$ is sectorial. So it remains to show that
$$\RR( \{t(t+\B)^{-1}; t>0\} ) < \infty$$
and that the $\RR$-angle of $\B$ is less or equal $\Phi$. We will need the following estimate. Let $\lambda \in S_{\pi-\Phi, 0} \subset\rho(\B)$ for arbitrary $x\in \Omega$. For any $x,y\in \Omega$ it holds
\begin{align*}
 \|\lambda(\lambda+ B_x)^{-1} - \lambda (\lambda+B_y)^{-1}\|_{\L(L_p(B))}&= \| (B_y-B_x) \lambda (\lambda+ B_x)^{-1} (\lambda+B_y)^{-1}\|_{\L(L_p(B))}
\leq \tilde{L}\| y-x\| \frac{M^2}{|\lambda|}.
\end{align*}
Here $\tilde{L}$ and $M$ are independent of $x,y$. The last inequality holds because $b_2(v)\in C^1(\overline{\Omega})$.\\[0.2cm]
Next let $(\Sigma, M, \mu)$ be a probability space and let $N\in \N$. For $j=1,\dots, N$ let $\epsilon_j$ be a random, $\{-1,1\}$-valued 
variable, let $U_j\in L_p(\Omega, L_p(B))=:X$ and let $\lambda_j \in S_{\pi-\Phi,0}$. Take further $q=p$. There exists a constant $C>0$ such that 
\begin{align*}
 \Big\| \sum_j \epsilon_j\lambda_j(\lambda_j+\B)^{-1}U_j\Big\|_{L_p(\Sigma,X)} &\leq \sum_j \|\epsilon_j \lambda_j(\lambda_j+\B)^{-1} U_j\|_{L_p(\Sigma,X)}\\
&\leq C \sum_j M \|U_j\| < \infty. 
\end{align*}
So the intergals exist and we can apply Fubini's Theorem. It holds
\begin{align*} 
 \Big\|\sum_{j=1}^N \epsilon_j \lambda_j(\lambda_j+\B)^{-1} U_j\Big\|^p_{L_p(\Sigma, X)}
= \int_\Omega \left( \int_\Sigma \Big\| \sum_j \epsilon_j(s) \lambda_j(\lambda_j+B_x)^{-1}\hat{U}_j(x)\Big\|_{L_p(B)}^p \, d\mu(s)\right) \, dx
 \end{align*}
where $\hat{U}$ in the integrals stands for a representative of $U\in L_p(\Omega, L_p(B))$.\\
(ii) From the previous Lemma and \cite{DHP} we know that $B_x$ is $\RR$-sectorial with $\RR$-angle smaller that $\Phi$ for any $x\in \overline{\Omega}$. So there is a constant $C(x)>0$ such that
\begin{equation}\label{Ungleich}
  \Big\|\sum_j \epsilon_j \lambda_j(\lambda_j+B_{x})^{-1} \hat{U}_j(x)\Big\|_{L_p(\Sigma, L_p(B))}\leq C(x)\Big\| \sum_j \epsilon_j \hat{U}_j(x)\Big\|_{L_p(\Sigma,L_p(B))},
\end{equation}
where $C(x)$ is the $\RR$-bound of $\{ \lambda (\lambda + B_x)^{-1}, \lambda \in S_{\pi-\Phi, 0}\}$.
Fix $x_0\in \Omega$. 
We get
\begin{align*}
&\Big\|\sum_j \epsilon_j \lambda_j(\lambda_j+B_x)^{-1} \hat{U}_j(x)\Big\|_{L_p(\Sigma, L_p(B))}\\
& \leq  \Big\|\sum_j \epsilon_j \lambda_j(\lambda_j+B_{x_0})^{-1} \hat{U}_j(x)\Big\|_{L_p(\Sigma, L_p(B))}\\
&\qquad+ \Big\|\sum_j \epsilon_j \left(\lambda_j(\lambda_j+B_x)^{-1}-\lambda_j(\lambda_j+B_{x_0})^{-1}\right) \hat{U}_j(x)\Big\|_{L_p(\Sigma, L_p(B))}.
\end{align*}
We estimate the second term with the help of \eqref{Ungleich}
\begin{align*}
&\Big\|\sum_j \epsilon_j \left(\lambda_j(\lambda_j+B_x)^{-1}-\lambda_j(\lambda_j+B_{x_0})^{-1}\right) \hat{U}_j(x)\Big\|_{L_p(\Sigma, L_p(B))}\\
& \leq C(x_0) \Big\| \sum_j \epsilon_j  (B_{x_0}-B_x) (\lambda_j+B_x)^{-1} \hat{U}_j(x)\Big\|_{L_p(\Sigma, L_p(B))}\\
& = C(x_0) \Big\| \sum_j \epsilon_j (B_{x_0}-B_x) B_x^{-1} (-\lambda_j+ B_x+\lambda_j) (\lambda_j+B_x)^{-1} \hat{U}_j(x)\Big\|\\
& \leq C(x_0) \left( \Big\|(B_{x_0}-B_x)B_x^{-1} \sum_j \epsilon_j  \lambda_j (\lambda_j+B_{x_0})^{-1} \lambda_j (\lambda_j+B_x)^{-1} \hat{U}_j(x)\Big\|
\right.\\
&\qquad\left.+ \Big\|(B_{x_0}-B_x)B_x^{-1} \sum_j \epsilon_j \lambda_j(\lambda_j+B_{x_0})^{-1} \hat{U}_j(x)\Big\|_{L_p(\Sigma, L_p(B))}\right)\\
& \leq \|(B_{x_0}-B_x)B_x^{-1}\|_{\L(L_p(B))} (C(x_0)C(x)+C(x_0) ) \Big\| \sum_j \epsilon_j \hat{U}_j(x)\Big\|_{L_p(\Sigma,L_p(B))}.
\end{align*}
Given $u\in L_p(B)$, the uniform a priori estimates for $B_x$ lead to
 $$\|B_x^{-1} u\|_{W_p^2} \leq C_p M_p \|B_xB_x^{-1}u\|_{L_p(\Omega)}= C_pM_p \|u\|_{L_p(\Omega)}. $$
  Further the map $x \mapsto B_x$ is Lipschitz continuous with some constant $\tilde{L}>0$. Since $C(x)$ was choosen minimal we conclude that
\begin{align}
 C(x) \leq C(x_0) (1+ L |x-x_0|) + L |x-x_0| C(x) C(x_0), \label{Rbounds}
\end{align}
for $L:= \tilde{L}C_p M$. The constant $C_p$ occurs if we consider the equivalent norm on $W_p^2(B)$ that only includes highest derivatives. Now we show that the constants $\{C(x),x\in \overline{\Omega}\}$ are unifomly bounded. Suppose instead,
$$\sup_{x\in  \overline{\Omega}} C(x) = \infty.$$
Then there exists a sequence $(x_n)_{n\in\N} \subset \overline{\Omega}$ such that $C(x_n) \rightarrow \infty$. Of course there is a subsequence of $(x_n)$ that converges to some $x\in \overline{\Omega}$. We call it again $(x_n)$.
Then we know that $C(x)< \infty$. Let $\epsilon >0$ and $n\in \N$ such that $|x_n-x| < \epsilon$. By (\ref{Rbounds}) with the identification of $x_0$ with $x$ and $x_n$ instead of $x$ it holds
$$ C(x_n) \leq C(x) (1+L\epsilon) + L \epsilon C(x) C(x_n).$$
The term $C(x)(1+L\epsilon)$ is bounded by some value $S$ if $\epsilon$ is small enough. So
\begin{align*}
C(x_n) &\leq L \epsilon C(x) C(x_n) + S.
\end{align*}
For $\epsilon = \frac{1}{2 L C(x)}>0$ we obtain a bound for $(C(x_n))_{n> N }$ which is a contradiction. \\
(iii) Let $C= \sup_{x\in \overline{\Omega}} C(x)$. Then with the considerations
of part (i) we conclude
\begin{align*}
 \Big\|\sum_{j=1}^N \epsilon_j \lambda_j(\lambda_j+\B)^{-1} U_j\Big\|^p_{L_p(\Sigma, X)}
&\leq C^p\Big\| \sum_j \epsilon_j U_j\Big\|^p_{ L_p(\Sigma, X)}.
\end{align*}
So $\B$ is $\RR$-sectorial with $\RR$-angle less or equal than $\Phi$. By the permanence properties for $\RR$-sectorial operators (\cite{DHP}, part 4.1) 
we conclude that $\A_2(v)$ is $\RR$-sectorial with the same $\RR$-angle. 
\end{proof}
Now Theorem 4.4 in \cite{DHP} 
ensures that $\A_2(v)$ has maximal $L_p$-regularity. For fixed $v\in C^1(\overline{\Omega})$ it is furthermore true that $A_1(v)$ possesses maximal $L_p$-regularity. These results are now transfered to the coupled operator. For $v\in C^1(\Omega)$ we set
\begin{align*}
dom(\A(v))& = \bigcup_{u\in W_p^2(\Omega)\cap W_p^{1,0}(\Omega) } \{u\} \times D_0(u),\\
\A(v) (u,U) &= \left( - div_x( a_1(v) \nabla_x u), \Phi_* [\mathcal{A}_x(v(x)) \Phi^* U] \right), && (u,U)\in dom(\A(v)).
\end{align*}
Note that the domain of definition $dom(\A(v))$ is indipendent of $v$. It is indeed $dom(\A)$. From the considerations before we know that there is a sector $S_{\pi-\Phi,0} \subset \rho(-\A_2(v))\cap\rho(-A_1(v))$. 
\begin{theorem}\label{MaxReg}
For any $v\in C^1(\overline{\Omega})$, the operator $\A(v)$ is $\RR$-sectorial and possesses maximal $L_p$-regularity.
\end{theorem}
\begin{proof}
Let $v\in C^1(\overline{\Omega})$ and  $\B= \Phi^* \A_2(v) \Phi_*$. 
Take $\lambda\in S_{\pi-\Phi,0} \subset \rho(-\B)\cap \rho(- A_1(v))$ and let 
$$(f,g)\in Y_0= L_p(\Omega) \times L_p(\Omega, L_p(\Omega_x))$$ 
be given. We set
\begin{align*}
 u &= (\lambda + A_1(v))^{-1} f,\\
 U &= \Phi_* (\lambda + \B)^{-1} \Phi^* ( g- \lambda Ru) + Ru,
\end{align*}
Then $(u,U)\in dom(\A(v))$ and it holds 
$$ (f,g) = (\lambda + \A(v))(u, U).$$
Then as in \cite{Dani1} we find a constant $M>0$ such that
\begin{align*}
 &|\lambda | \| (\lambda + \A(v))^{-1}\|_{\L(L_p(Y_0))}\\ &\leq |\lambda| \| (\lambda + A_1(v))^{-1}\| + |\lambda| \|\Phi_* (\lambda + \B)^{-1}\Phi^*\| (1+ |\lambda | \| R\| \| (\lambda + A_1(v))^{-1}\|)\\
&\qquad + (1+|\lambda| \|R\| \|(\lambda + A_1(v))^{-1}\|) |\lambda | \|\Phi_* (\lambda+\B)^{-1} \Phi^*\|\leq M. 
\end{align*}
Hence $\A(v)$ is sectorial. Take $(\Sigma, M, \mu)$, $N\in\N$ and 
$\epsilon_j, \lambda_j$ as in Proposition \ref{Hilfe2}. For $j=1,\dots, n$ let $u_j\in L_p(\Omega)$, $U_j\in L_p(\Omega,L_p(\Omega_x))=X$. We use again Fubini's Theorem and the methods of the proof for $\B$. Let $C_1, C_2$ be the bounds from the $\RR$-calculus for $A_1(v)$ and $\B$. We get
\begin{align*}
& \Big\| \sum_{j=1}^N \epsilon_j \lambda_j (\lambda_j + \A(v))^{-1} (u_j, U_j)\Big\|_{L_p(\Sigma,Y_0)}\\
&\leq C_1(1+\|R\|) \Big\|\sum_j \epsilon_j u_j\Big\|_{L_p(\Sigma,L_p(\Omega))} + C_2 \|\Phi_*\| \Big\| \sum_j \epsilon_j \Phi^* U_j\Big\|_{L_p(\Sigma,X)}\\
&\qquad + \left(C_2+1\right)C_1 \|\Phi_*\| \|\Phi^*\| \|\B^{-1}\|\|R\| \Big\|\sum_j \epsilon_j u_j\Big\|_{L_p(\Sigma,L_p(\Omega))}\\
&\leq C\Big\| \sum_j \epsilon_j(u_j,U_j)\Big\|_{L_p(\Sigma,Y_0)}.
\end{align*}
Now the assertion follows from \cite{DHP}.
\end{proof}
We will write $\A(v) \in MR(p,Y_0)$ meaning that $\A(v)$ possesses maximal $L_p$-regularity with respect to $Y_0 = L_p(\Omega)\times L_p(\Omega, L_p(\Omega_x))$.
Take $n+2<p<\infty$. Then due to Sobolev embeddings it holds
$$W_p^{2-\frac2p}(\Omega) \hookrightarrow C^1(\overline{\Omega}).$$
The definition of $\A(v)$ and the characterization of the interpolation space $Y_{1-\frac1p,p}$, cf. \eqref{Interpol}, ensure that
$$\big(v\mapsto \A(v)\big) \in C^{1-}\left( Y_{1-\frac1p,p}, \L(D(\A),Y_0)\right).$$
Local existence of the solution for the quasilinear initial-boundary value problem now follows from a general result due to Cl\'ement and Li.
\begin{Cor}\label{main}
 Let $p>n+2$, let $\A(\cdot)$ be defined as above, $T_0>0$ and 
\begin{align*}
 f &\in C^{1-,1-} \left([0,T_0]\times Y_{1-\frac1p,p}, Y_0\right), & g &\in L_p([0,T_0], Y_0).
\end{align*}
Let $(u_0,U_0)\in Y_{1-\frac1p,p}$. Then there exists $T_1\in (0,T_0]$ and unique functions $$(u,U)\in L_p((0,T_1), D(\A))\cap W_p^1((0,T_1), Y_0)\cap C([0,T_1], Y_{1-\frac1p,p})$$
that satisfy
\begin{align*}
\begin{cases}
 \left(\dot{u}, \dot{U}\right) + \A(u(t))(u(t),U(t)) = f\big(t, u(t), U(t)\big) + g(t), \qquad \text{on }(0,T_1),\\
\qquad \qquad \qquad \quad \;(u(0), U(0))=(u_0,U_0).
\end{cases}
\end{align*}
\end{Cor}
\begin{proof}
 By Theorem \ref{MaxReg} it holds
$$ \A(u_0)\in MR(p,Y_0).$$ So the assumptions of Theorem 2.4 in \cite{ClLi} are fulfilled and the assertion follows.
\end{proof}
\section{An ansatz for the dependence on $U$}\label{reaction}
In part 2 the operator $\A$ was allowed to depend on the macroscopic density $u$. But usually the homogenisation process also produces a dependence on the micro density $U$. 
We try to circumvent this problem and give a possible way to include this dependence as well. Let $u, U$ denote the concentration of a contaminant in the macro and micro scale of a double porosity system. With $m, M$ we denote the two porosities. If the diffusivity in both scales depends on the concentration $u$ then this can be modeled by using the Nemytskii operators $a_1,b_2$ introduced in the last chapter. 
In general the solute can interact with the  solid structure. Assume that it gets attached to the walls with a certain rate depending on $u$. All different effects like reaction, van-der-Waals forces, electric forces etc. which bind the solute, are 
merged in one sorption term. The concentration of the sorped contaminant is then given by a function $u^*$. 
We suppose that $u^*$ depends linearly on the concentration $u$ in the fluid. Thus
$$\partial_tu^* = K \partial_t u.$$
We neglect any dependence of the sorption on the changing porosity. The interaction results in a prefactor in our equations. So for $t\in (0,T]$ we get
\begin{align*}
       (1+K)\partial_t (mu) - div_x(m a_1(u) \nabla_x u) =& f_1(x,t,u, U), & \text{on } \Omega,\\
	(1+K)\partial_t (MU) - div_z(M a_2(u) \nabla_z U) =& f_2(x,t,u, U), & \text{on } Q.
\end{align*}
Now we assume that the reaction has an effect on the solid structure of the porous blocks i.e. it changes $M$. Since the fissures are of another, larger length scale we do not consider a change of $m$ and set it to one. We define an average and possibly evolving porosity of each block. Let $M:\Omega\times[0,T] \to (0,\infty]$. Moreover we assume that $M$ has values in a bounded interval not including zero. This excludes the case of total pore 
closure in the cells and also that large parts of the solid structure dissolve in the fluid.  So $M(x)$ is bounded between some minimal $M_{min}>0$ and a maximal value $M_{min}<M_{max}<\infty$.  
Let $U\in L_p(\Omega, L_p(\Omega_x)), V= \Phi_* U$. The total amount of reactant in the cell is approximated by 
$$|\Phi_x(B)| \int_{B} \hat{V}(x,y) \, dy.$$
We suppose that there is a function $G_0$, 
$$G_0: [0,\infty) \times [M_{min}, M_{max}] \to [0, C_G],$$
which describes the change of the porosity due to the reaction. It depends on the amount of contaminant in the cell and the porosity. We assume that $G_0$ is Lipschitz continuous in both variables. 

We assume that $G_0$ has to have the same form in all cells. Thus the same Lipschitz constant and bounds are valid for arbitrary $x\in \Omega$. We define 
\begin{align*}
 G: L_p(\Omega,L_p(\Omega_x))\times L_p(\Omega) &\to L_p(\Omega),\\
(U,M) &\mapsto \left(x\mapsto G_0\left(|\Phi_x(B)| \int_{B} \hat{V}(x,y) \, dy, M(x)\right)\right).
\end{align*}
Since $G_0$ is continuous and the arguments of $G$ are measurable on $\Omega$
this is a measurable function on $\Omega$.
In addition we assume that there exists $c_g >0$ such that
\begin{align}
 \sup_{U,M}\, |G(U,M)|< c_g.\label{bound}
\end{align}
The supremum is taken over all admissible $U$ and $M$.  
Finally let $M_0 \in L_p(\Omega)\cap L_\infty(\Omega)$ be admissible. We suppose that for fixed $U\in L_p(\Omega, W_p^2(\Omega_x))$ the evolution equation
\begin{align}
 \left\{\begin{array}{rl}
  \partial_t M(t) &= -G(U,M), \\
  M(0)&=M_0
 \end{array}\right.\label{porous}
\end{align}
has an unique solution in $L_p((0,T_0), L_p(\Omega))\cap W_p^1((0,T_0), L_p(\Omega))$. 
Now we transform the resulting system of equations and write all terms which include $M$ or its derivative on the right hand side. Thus for some $0<T_1<T_0$, the functions $(u,U)$ shall satisfy
\begin{align}\label{P_var}
\left\{\begin{array}{rl}
  \left(\dot{u}(t), \dot{U}(t))\right) + \A(u(t))(u(t),U(t)) &= \tilde{f}\big(t, u(t), U(t), M(t)\big), \, \text{on }(0,T_1),\\
(u(0),U(0))&=(u_0,U_0),\\
\dot{M}(t) &= -G(U,M),\qquad\qquad \quad \text{on }(0,T_1),\\
M(0) &= M_0.
\end{array}\right.
\end{align}
Here $\tilde{f}$ is defined by
\begin{align*}
 \tilde{f}(t,u,U, M) &= \left( f_1(t,u,U), \frac1M f_2(t,u,U) - \partial_t M \frac{U}{M}\right).
\end{align*}
Let us first summerize our notation. We use
\begin{align*}
 Y_0 &= L_p(\Omega)\times L_p(\Omega,L_p(\Omega_x)), & Y_1&= dom(\A),\\
Y_{1-\frac1p} &= (Y_0,Y_1)_{1-\frac1p,p}. &E_1 &= L_p(\Omega), \qquad E_2 = L_p(\Omega,L_p(\Omega_x)).  
\end{align*}
Further we write
\begin{align*}
 X^T &= L_p(0,T; Y_0),&& E_1^T = L_p(0,T; E_1),\\
 Y^T &= W_p^1(0,T;Y_0)\cap L_p(0,T; Y_1),&&E_2^T = L_p(0,T; E_2),\\
 Z^T &= \{\vec{u}\in Y^T; \vec{u}(0)=0\}.
\end{align*}
With $\vec{u}$ we denote  a pair of functions $(u,U)$ from $Y_0$ or $Y^T$. 
Now we can formulate the well posedness result.
\begin{theorem}\label{Ende}
 Suppose $f=\left(f_1,f_2\right)\in C^{1-,1-}([0,T_0)\times Y_{1-\frac1p,p},Y_0)$ for some $T_0>0$ and $G_0, G$ are given functions
as above. For all $(v,V) \in Y_{1-\frac1p,p}$ let $\A(v)$ be defined as before. Let $(u_0, U_0)\in Y_{1-\frac1p,p}, M_0 \in C(\Omega)$, $M_{min}<M_0<M_{max}$. Assume that
\begin{align*}
 \|f_2(t, \vec{u}_0)\|_{E_2} &\leq \frac{M_{min}^2}{4},&&\text{for all } t\in [0,T_0],\\
 \|U_0\|_{E_2} &\leq \frac{M_{min}^2}{4c_g}. 
\end{align*}
 Then there  exists  $T_1(\vec{u}_0)\in (0,T_0]$ and unique functions $$(u,U) \in L_p(0,T_1; Y_1) \cap W_p^1(0,T_1;Y_0)\cap C([0,T_1], Y_{1-\frac1p,p})$$ and $M\in W_p^1(0,T_1; E_0)$ that satisfy (\ref{P_var}).
\end{theorem}
\begin{proof} In this proof we write $v'$ instead of $\dot{v}$ to denote the time derivative. Let $\vec{w},\tilde{M}$ be the solutions of the following linear problems. Let $\tilde{M}\in W_p^1(0,T_0; E_1)$ be the solution of 
\begin{align*}
 \tilde{M}'(t) &= -G(U_0,\tilde{M}),&&t\in (0,T_0),\\
 M(0)&= M_0.
\end{align*}
The solution exists due to our assumptions (\ref{porous}).
Let $\vec{w}$ satisfy
\begin{align*}
 \vec{w}'(t) + \A(u_0) \vec{w}(t) &= \tilde{f}(t,\vec{u}_0, \tilde{M}), && t\in (0,T_0),\\
\vec{w}(0)&= \vec{u}_0.
\end{align*}
Because of the properties of $\A(u_0)$ it is well defined.
For the rest of the proof we write $f$ instead of $\tilde{f}$. 
Given $0< T<T_0$, $\rho>0$, let
\begin{align*}
 \Sigma_{\rho,T} = \Big\{(\vec{u},M)\in Y^T\times E_1^T; \vec{u}(0)=\vec{u}_0,\, M(0)=M_0,
\|\vec{u}-\vec{w}\|_{Y^T} \leq \rho, \, \| M-\tilde{M}\|_{E_1^T} \leq &\rho\Big\}.
\end{align*}
We follow the steps of the proof of Theorem 2.1 in \cite{ClLi}. There exists $T_2 >0$ small enough, such that
$$\left\{ \vec{u}(t); t\in (0,T], \vec{u}\in \Sigma_{\rho,T}\right\} \subset Y_{1-\frac1p}$$
if $T\in (0, T_2]$. So for those $\vec{u}$ maximal regularity of $\A(u(t))$ is preserved in this time interval. We deduce that there is a constant $L>0$, such that for $\vec{u}_1,\vec{u}_2\in \Sigma_{\rho,T}$, $t\in (0,T)$, it holds
\begin{align}
 \|\A(u_1(t))-\A(u_2(t))\|_{\mathcal{L}(Y_1,Y_0)} \leq L \|\vec{u}_1(t)-\vec{u}_2(t)\|_{Y_{1-\frac1p}}\label{CL2-9} \\
\|f_i(t,\vec{u}_1(t))-f_i(t,\vec{u}_2(t))\|_{E_i} \leq L\|\vec{u}_1(t)-\vec{u}_2(t)\|_{Y_{1-\frac1p}},&& i=1,2.\label{CL2-10}
\end{align}
We define the mapping 
$$\gamma : Y^T\times E_1^T\to Y^T \times E_1^T:\gamma(\vec{u},M) = (\vec{v},N),$$ 
where $(\vec{v},N)$ is the unique solution of the linear problem 
\begin{align*}
 \vec{v}'(t) + \A(u_0) \vec{v}(t) &= \A(u_0) \vec{u}(t) - \A(u(t)) \vec{u}(t) + f(t,\vec{u}(t), M(t)), && t\in (0,T),\\
\vec{v}(0) &= \vec{u}_0,\\
N'(t) &= -G(U(t),M(t)),&& t\in (0,T),\\
N(0)&= M_0.
\end{align*}
To apply Banach's fixed point theorem we show that $\gamma$ is a contraction on some $\Sigma_{\rho_1,T_1}$. So we estimate $\|\vec{v}-\vec{w}\|_{Z^T}$, $\|M-\tilde{M}\|_{E_1^T}$ and $\|\gamma(\vec{u}_1,M_1)-\gamma(\vec{u}_2,M_2)\|$. Let $(\vec{u},M)\in \Sigma_{\rho,T}$ and $(\vec{v},N)= \gamma(\vec{u},M)$. Then it holds for $t\in (0,T)$,
\begin{align*}
 (\vec{v}-\vec{w})'(t) + \A(u_0)(\vec{v}-\vec{w})(t) &= \A(u_0)\vec{u}(t) - \A(u(t)) \vec{u}(t)+ f(t,\vec{u}(t),M(t))-f(t,\vec{u}_0,\tilde{M}(t)),\\
(\vec{v}-\vec{w})(0)&= 0.
\end{align*}
Now the maximal regularity of $\A(u_0)$ allows us to apply Corollary 2.3 from \cite{ClLi}. Thus there is a constant $\M>0$, such that
\begin{align}
 \left\| \left( \frac{d}{dt} + \A(u_0)\right)^{-1}\right\|_{\L(X^T,Z^T)} &\leq \M,\label{H1}\\
\|\vec{v}\|_{C([0,T], Y_{1-\frac1p})} &\leq \M \|\vec{v}\|_{Z^T} &&\text{for } \vec{v}\in Z^{T}.\label{H2} 
\end{align}
So with (\ref{CL2-9}) we get
\begin{align*}
 \|\vec{v}-\vec{w}\|_{Z^T} 
&\leq \M L (\M\rho +\|\vec{w}-\vec{u}_0\|_{C([0,T],Y_{1-\frac1p}})(\rho+\|\vec{w}\|_{Y^T})+\M \|f(t,\vec{u}, M)-f(t,\vec{u}_0,\tilde{M})\|_{X^T}.
\end{align*}
With (\ref{CL2-10}) and the assumptions on $G$ we estimate the last term
\begin{align*}
 \|f(t,\vec{u}, M)-f(t,\vec{u}_0,\tilde{M})\|_{X^T}&\leq \|f_1(t,\vec{u})-f_1(t,\vec{u}_0)\|_{E_1^T}+ \left\| \frac{1}{M} f_2(t,\vec{u})-\frac{1}{\tilde{M}} f_2(t,\vec{u}_0)\right\|_{E_2^T}\\
& \qquad   + \left\|\frac{U}{M}G(U,M) -\frac{U_0}{\tilde{M}}G(U_0,\tilde{M}) \right\|_{E_2^T}.
\end{align*}
It holds
\begin{align*}
  \|f_1(t,\vec{u})-f_1(t,\vec{u}_0)\|_{E_1^T} &\leq L\|\vec{u}-\vec{u}_0\|_{L_p((0,T), Y_{1-\frac1p})}\leq LT^{1/p} (\M\rho + \|\vec{w}-\vec{u}_0\|_{C([0,T],Y_{1-\frac1p}}).
\end{align*}
The second part is calculated in a similar way. We find
\begin{align*}
 & \left\| \frac{1}{M} f_2(t,\vec{u})-\frac{1}{\tilde{M}} f_2(t,\vec{u}_0)\right\|_{E_2^T} + \left\|\frac{U}{M}G(U,M) -\frac{U_0}{\tilde{M}}G(U_0,\tilde{M}) \right\|_{E_2^T}\\
&\leq \frac{1}{M_{min}^2} \Big[ \| \tilde{M} \|_{E_1^T} \|f_2(t,\vec{u})-f_2(t,\vec{u}_0)\|_{C([0,T], E_2)}  + \|\tilde{M}-M\|_{E_1^T}\|f_2(t,\vec{u}_0)\|_{C([0,T],E_2)}\\
&\quad+  \|\tilde{M}\|_{E_1^T} 2c_g \|\vec{u}-\vec{u}_0\|_{C([0,T], Y_{1-\frac1p})}+ \|\tilde{M}-M\|_{E_1^T} \| U_0 G(U_0, \tilde{M})\|_{C([0,T],E_2)}  \Big]. 
\end{align*}
Observe that $\|f_2(t,\vec{u}_0)\|\leq \frac{M_{min}^2}{4}$ and $\| U_0 G(\vec{u}_0, \tilde{M}(t))\|\leq \frac{M^2_{min}}{4}$. The values $\|\tilde{M}\|_{E_1^T}$ and $\Psi(T) :=\|\vec{w}-\vec{u}_0\|_{C([0,T], Y_{1-\frac1p})}$  vanish if $T\to 0$. Further
$$ \|f_2(t,\vec{u})-f_2(t,\vec{u}_0)\|_{C([0,T], E_2)} \leq L\|\vec{u}-\vec{u}_0\|_{C([0,T], Y_{1-\frac1p})}\leq L(\M\rho + \Psi(T)).$$
We summarize using $\|\tilde{M}-M\|_{E_1^T} \leq \rho$, 
\begin{align}\label{Bed1}
 \|\vec{v}-\vec{w}\|_{Z^T} &\leq \M L (\M\rho + \Psi(T))(\rho +\|\vec{w}\|_{Y^T})+ \M LT^{1/p}(\M\rho+\Psi(T))\nonumber \\
&\quad + \frac{1}{M_{min}^2} \Big( \M \|\tilde{M}\|_{E_1^T} L (\M\rho + \Psi(T))  + 2c_g\|\tilde{M}\|_{E_1^T} (\M \rho+\Psi(T))\Big)+ \frac12\rho .
   \end{align}
The values $T$ and $\rho$ can be choosen in the way, such that the right hand side is smaller than $\rho$. Now the difference between $N$ and $\tilde{M}$ remains to be considered. For $t\in(0,T]$ we have
\begin{align*}
 (N-\tilde{M})(t) = \int_0^t (N'(s)-\tilde{M}'(s)) \, ds = \int_0^t (G(U_0,\tilde{M})-G(U,M))(s) \, ds.
\end{align*}
Hence using Corollary 2.3 from \cite{ClLi} and (\ref{bound}) we get
\begin{align*}
\| N-\tilde{M}\|_{E_1^T} 
&\leq 2c_g (1+p)^{-\frac1p} T^\frac{p+1}{p}. 
\end{align*}
So for 
\begin{equation}\label{Bed2}
     T< \left(\frac{1+p}{(2 c_g)^p}\right)^{\frac{1}{p+1}} \rho^{\frac{p}{p+1}} 
   \end{equation}
this is smaller than $\rho$. 
In the next step we show that $\gamma$ is a contraction. Let $$(\vec{u}_i,M_i)\in \Sigma_{\rho,T}, \qquad \gamma(\vec{u}_i,M_i) = (\vec{v}_i,N_i)\quad \text{for }i=1,2.$$ It holds for $t\in (0,T)$,
\begin{align*}
 (\vec{v}_1-\vec{v}_2)'(t) + \A(u_0)(\vec{v}_1(t)-\vec{v}_2(t)) &= \A(u_0)(\vec{u}_1(t)-\vec{u}_2(t))- \A(u_1(t)) \vec{u}_1(t)  + \A(u_2(t)) \vec{u}_2(t) 
\\&\quad+ f(t,\vec{u}_1,M_1) -f(t,\vec{u}_2,M_2),\\
 (\vec{v}_1-\vec{v}_2)(0)&= 0.
\end{align*}
Similar arguments as above show that
\begin{align*}
 \|\vec{v}_1-\vec{v}_2\|_{Z^T} 
&\leq \M L \|\vec{u}_1-\vec{u}_2\|_{Z^T}(3\M\rho+\|\vec{w}\|_{Y^T})+\M \|f(t,\vec{u}_1, M)-f(t,\vec{u}_2,\tilde{M})\|_{X^T}.
\end{align*}
Again we treat the last term separately. It holds
\begin{align*}
 \|f(t,\vec{u}_1, M_1)-f(t,\vec{u}_2,M_2)\|_{X^T}
&\leq \|f_1(t,\vec{u}_1)-f_1(t,\vec{u}_2)\|_{E_1^T} + \left\| \frac{1}{M_1} f_2(t,\vec{u}_1)-\frac{1}{M_2} f_2(t,\vec{u}_2)\right\|_{E_2^T}
\\& \quad+ \left\|\frac{U_1}{M_1}G(U_1,M_1) -\frac{U_2}{M_2}G(U_2,M_2) \right\|_{E_2^T}.
\end{align*}
We have
$$ \|f_1(t,\vec{u}_1)-f_1(t,\vec{u}_2)\|_{E_1^T} \leq L\|\vec{u}_1-\vec{u}_2\|_{L_p((0,T), Y_{1-\frac1p})}\leq \M LT^{1/p} \|\vec{u}_1-\vec{u}_2\|_{Z^T}.$$
With the use of $\vec{w}$ and the assumptions on $G$ we conclude (all norms here are in $E_2^T$) 
\begin{align*}
 & \left\| \frac{1}{M_1} f_2(t,\vec{u}_1)-\frac{1}{M_2} f_2(t,\vec{u}_2)\right\| + \left\|\frac{U_1}{M_1}G(U_1,M_1) -\frac{U_2}{M_2}G(U_2,M_2) \right\|\\
&\leq \frac{1}{M_{min}^2} \Big[ \|M_2-M_1\|_{E_1^T}( \M L \rho +\frac{M_{min}^2}{4}+ L\Psi(T))  + \|\vec{u}_1-\vec{u}_2\|_{Z^T}(\rho \M L + \M L \|\tilde{M}\|_{E_1^T})\\&\quad + \| M_2U_1 G(U_1,M_1)-M_1U_2 G(U_2,M_2)\| \Big].
\end{align*}
Further
\begin{align*}
 &\| M_2U_1 G(U_1,M_1)-M_1U_2 G(U_2,M_2)\|_{E_2^T}\\
&\leq \M 2c_g\|\vec{u}_1-\vec{u}_2\|_{Z^T} (\rho + \|\tilde{M}\|_{E_1^T}) 
\quad+ \|M_2-M_1\|_{E_1^T}(2c_g (\M \rho+ \Psi(T))  + \|U_0 G(U_0,\tilde{M})\|_{C([0,T],E_2)}). 
\end{align*}
Finally the continuity of $G$ implies
\begin{align*}
 \|N_1-N_2\|_{E_1^T} \leq T^\frac1p c_p c(\|\vec{u}_1-\vec{u}_2\|_{Z^T}+ \|M_1-M_2\|_{E_1^T}).
\end{align*}
Because of the assumptions there exist $(\rho_1,T_1)$, such that (\ref{Bed1}) is smaller than $\rho_1$, (\ref{Bed2}) is satisfied and it holds 
$$\|\gamma(\vec{u}_1,M_1) -\gamma(\vec{u}_2,M_2)\|_{\Sigma_{\rho_1 T_1}} \leq \frac34 \|(\vec{u}_1,M_1)-(\vec{u}_2,M_2)\|_{\Sigma_{\rho_1 T_1}}.$$
Now Banach's fixed point theorem proves the assertion.
\end{proof}

\ \\
\
\hspace{1cm}
\tt
Institute for Applied Mathematics,
Leibniz University Hanover,
Welfengarten 1,\\
D 30167 Hanover,
Germany\\

\end{document}